\documentclass[11pt]{amsart} 
\usepackage{amsmath, amssymb, amsthm, mathtools, caption}
\usepackage{enumitem}
\usepackage{float}
\usepackage{listings}
\usepackage{xcolor}
\definecolor{codegray}{gray}{0.95}
\usepackage[noabbrev,capitalize]{cleveref}
\usepackage{adjustbox}
\crefname{equation}{}{}
\usepackage{ytableau}
\usepackage{graphicx, tikz}
\usepackage[textheight=8.80in, textwidth=6.85in]{geometry}

\usepackage{microtype}
\usepackage{bbm}

\newtheorem{theorem}{Theorem}
\newtheorem{lemma}[theorem]{Lemma}
\newtheorem{corollary}[theorem]{Corollary}

\newtheorem*{conjecture*}{Conjecture}

\theoremstyle{definition}

\theoremstyle{remark}
\newtheorem*{remark}{Remark}

\newtheorem*{example}{Example}

\numberwithin{equation}{section}

\usepackage{bm} 

\newcommand{\qbinom}[2]{\genfrac{[}{]}{0pt}{}{#1}{#2}}

\title[Hermite-Jensen limits and $d$ log-concavity of $q$-multinomials]{Hermite-Jensen limits and $d$ log-concavity of $q$-multinomials}


\author{Ken Ono}

\address{Dept. of Mathematics, University of Virginia, Charlottesville, VA 22904, USA}
\email{ko5wk@virginia.edu}

\begin{document}
\maketitle

\begin{abstract}
In 1878, Sylvester proved Cayley's Conjecture that the coefficients of the Gaussian $q$-binomial coefficients are unimodal. In 1990, O'Hara famously discovered a constructive combinatorial proof, and in 2013, Pak and Panova proved the stronger property of strict unimodality for sufficiently large parameters. We move from unimodality to log-concavity and higher degree $ d$ log-concavity, known as Tur\'an inequalities. Although $q$-binomial coefficients are not always log- or degree $d$ log-concave, it's natural to ask to what extent these inequalities hold. In infinite families with limiting ``aspect ratio" bounded away from zero and one, we prove that these stronger inequalities hold uniformly, for each $C>0,$ on the central window
$|m-\mu|< C\sigma,$ where $\mu$  and $\sigma$ are the mean and standard deviation of the normalized distribution. More generally, we obtain the same conclusions for $q$-multinomial coefficients. These results stem from the asymptotic behavior of normalized Jensen polynomials, which are approximated by Hermite polynomials.
\end{abstract}

\section{Introduction and Statement of Results}

In this note, we revisit and refine the classical story about the unimodality of the Gaussian (or $q$-binomial) coefficients.
For $a\in\mathbb{C}$ and $n\in\mathbb{Z}_{\ge0}$, we have the $q$-Pochhammer symbol 
\[
(a;q)_n := \prod_{j=0}^{n-1} (1-a q^{\,j}),\qquad (a;q)_0:=1.
\]
For integers $a, b\geq 0,$ the {\it Gaussian (or $q$-binomial) coefficient} is the palindromic polynomial
\begin{equation}
\qbinom{a+b}{a}_q
:= \frac{(q;q)_{a+b}}{(q;q)_a\,(q;q)_b}
= \sum_{k=0}^{ab} c_{a,b}(k)\, q^k.
\end{equation}
The coefficients $c_{a,b}(k)$ count partitions that fit in an $a\times b$ rectangle.
For concreteness, we offer the following examples
\[
\begin{split}
\qbinom{4}{1}_q &= 1 + q + q^2 + q^3,\\[2pt]
\qbinom{4}{2}_q &= 1 + q + 2q^2 + q^3 + q^4,\\[2pt]
\qbinom{5}{2}_q &= 1 + q + 2q^2 + 2q^3 + 2q^4 + q^5 + q^6,\\[2pt]
\qbinom{6}{3}_q &= 1 + q + 2q^2 + 3q^3 + 3q^4 + 3q^5 + 3q^6 + 2q^7 + q^8 + q^9.
\end{split}
\]

In 1856, Cayley famously conjectured \cite{Cayley1856} the unimodality of the coefficients of $q$-binomial coefficients, where a finite sequence $(a_k)_{k=0}^N$ is \emph{unimodal} if there exists an index $m$ such that
\[
a_0 \le a_1 \le \cdots \le a_m \ge a_{m+1} \ge \cdots \ge a_N,
\]
with ``strictly unimodal'' meaning all the displayed inequalities are strict.
In 1878, Sylvester \cite{Sylvester1878} gave the first proof of this conjecture using invariant theory.  
More than a century later, O'Hara solved the famous open problem of giving a \emph{constructive combinatorial} proof of unimodality via an explicit bijection  \cite{OHara}.  To be more precise, she defined, for each $0\le k<ab/2$, an explicit injection
\[
\Phi_k:\ \bigl\{\lambda\subseteq a\times b:\ |\lambda|=k\bigr\}\ \hookrightarrow\ \bigl\{\lambda\subseteq a\times b:\ |\lambda|=k+1\bigr\},
\]
built from local row/column moves on Ferrers diagrams (and using conjugation symmetry for $k>ab/2$), thereby showing $c_{a,b}(k)\le c_{a,b}(k+1)$ up to the middle and hence unimodality of $\qbinom{m+n}{m}_q$.
Subsequently, strict unimodality for large rectangles was established by Pak and Panova. Namely, for  $m,n\ge 8$, the sequence $\{c_{m,n}(k)\}_{k=0}^{mn}$ increases \emph{strictly} up to the middle and then decreases strictly \cite{PakPanova}. 
They proved strict unimodality by recasting the successive differences $c_{m,n}(k+1)-c_{m,n}(k)$ in terms of Kronecker coefficients in the representation theory of symmetric groups, whose positivity is guaranteed by the semigroup property and stability.

Unimodality is often proved using the stronger notion of log-concavity.
A finite sequence $(a_k)_{k=0}^N$ of nonnegative reals is \emph{log-concave} if, for all $1\le k\leq N-1$, we have
\begin{equation}\label{eq:LC}
a_k^2 \;\ge\; a_{k-1}\,a_{k+1}.
\end{equation}
For non-negative sequences, (1.2) implies unimodality. Indeed, if we set
\[
r_k:=\frac{a_{k+1}}{a_k},
\]
then
\[
a_k^2\ge a_{k-1}a_{k+1}
\qquad\Longleftrightarrow\qquad
r_{k-1}\ge r_k,
\]
so $(r_k)$ is nonincreasing. Choose $m$ such that
\[
r_{m-1}\ge 1\ge r_m
\]
(take $m=0$ if all $r_k\le 1$, and $m=N$ if all $r_k\ge 1$). Then
\[
a_0\le \cdots \le a_m \ge \cdots \ge a_N,
\]
so the sequence is unimodal. If zeros occur, delete the initial and terminal zero blocks
(which preserves unimodality) and apply the positive case to the remaining segment.

It turns out that $q$-binomial polynomials are not globally log-concave. For example, we have
\[
\qbinom{4}{2}_q \;=\; 1 + q + 2q^2 + q^3 + q^4,
\]
so the row $(1,1,2,1,1)$ violates \eqref{eq:LC} at $k=1$ since $1^2<1\cdot 2$.
This raises a natural question: \emph{to what extent} does log-concavity hold?
For example, does log-concavity hold for ``windows'' of coefficients in $q$-binomial coefficients?
We prove that for ``balanced'' parameters this is the case.
Moreover, we show that this phenomenon holds for the more general notion of $d$ log-concavity, which fits into the hierarchy of
Tur\'an/Laguerre-type inequalities.
For a sequence $a=(a_k)$, we consider the operator
\begin{equation}
(\mathcal{L}a)_k \;:=\; a_k^2 - a_{k-1}a_{k+1}\qquad(1\le k\le N-1),
\end{equation}
and define iterates $\mathcal{L}^d$ recursively.
We say $a$ is \emph{$d$ log-concave on a set of indices $\mathcal{W}\subseteq\{0,\dots,N\}$} if
\begin{equation}
(\mathcal{L}^r a)_k \;\ge\; 0\quad\text{for all }k\in\mathcal{W}\text{ and all }1\le r\le d
\end{equation}
(where $a_{-1}=a_{N+1}=0$).
Log-concavity is the $d=1$ case of the higher Tur\'an inequalities. 

Here we establish the Tur\'an inequalities for ``windows'' of coefficients using the theory of Jensen and Hermite polynomials, following the strategy first described in \cite{GORZ} in connection with the Riemann Hypothesis.
To this end, suppose that $a$ and $b$ are positive integers, and let
\begin{equation}
\qbinom{a+b}{a}_q=\sum_{k=0}^{ab} c_{a,b}(k)\,q^k \qquad {\text {and}}\qquad
p_{a,b}(k):=\frac{c_{a,b}(k)}{\binom{a+b}{a}},
\end{equation}
and we let
\begin{equation}
\mu_{a,b}:=\frac{ab}{2} \qquad {\text {and}} \qquad
\sigma_{a,b}^2:=\frac{ab(a+b+1)}{12}.
\end{equation}
These parameters have a natural probabilistic interpretation. Indeed, let
$K$ be the random index with $\Pr[K=k]=p_{a,b}(k)$. Then we have
\[
\mu_{a,b}=\mathbb E[K]\qquad\text{and}\qquad \sigma_{a,b}^2=\operatorname{Var}(K),
\]
as proved in Lemma~\ref{lem:mu-sigma} below.
To see that the mean in $\mu_{a,b}=ab/2$, we note that the palindromicity
$c_{a,b}(k)=c_{a,b}(ab-k)$ (i.e. the distribution is symmetric about the midpoint $ab/2$).
The variance $\sigma_{a,b}^2=ab(a+b+1)/12$ is the second central moment; it is obtained by a
standard differentiation of $\log\qbinom{a+b}{a}_q$ at $q=1$.

For a sequence $(u_k)_{k\in\mathbb Z}$, integers $m\in\mathbb Z$ and $d\ge0$, the degree $d$
\emph{index $m$ Jensen polynomial} is
\begin{equation}
J^{d,m}(X;u):=\sum_{j=0}^{d} \binom{d}{j} u_{m+j}\,X^{j}.
\end{equation}
We apply this to the probability weights $u_k=p_{a,b}(k)$. 
Namely, we define the {\it normalized Jensen polynomials} by
\begin{equation}\label{JensenNormal}
\mathcal{J}^{d,m}_{a,b}(X)
:=\frac{\delta_{a,b}^{-d}}{p_{a,b}(m)}\,
J^{d,m}\!\big(\,\delta_{a,b} X - 1;\,p_{a,b}\big).
\end{equation}
where $\delta_{a,b}:=\frac{1}{\sqrt{2}\sigma_{a,b}}.$
We compare these to the physicists'-style Hermite polynomials  defined by the
generating function
\begin{equation}\label{eq:GORZ-Hermite}
\sum_{d=0}^{\infty} H_d(X)\,\frac{t^d}{d!} \;=\; e^{-t^2+Xt}=1+Xt+(X^2-2) \frac{t^2}{2!}+
(X^3-6X) \frac{t^3}{3!}+\dots
\end{equation}
We prove the following theorem.

\begin{theorem}\label{thm:HermiteLimit}
Fix $d\ge1$ and $\lambda\in(0,1)$, and suppose $a,b\to +\infty$ with $a/(a+b)\to\lambda$. Then for every $C>0$, uniformly for integers $m$ with $|m-\mu_{a,b}|\le C\,\sigma_{a,b}$, coefficientwise we have
\[
\mathcal J^{d,m}_{a,b}(X)=H_d(X)+O_{d,\lambda,C}\big((a+b)^{-1/2}\big).
\]
\end{theorem}

\begin{example}
If we let $(a,b)=(50,50)$, then we have
\[
\qbinom{100}{50}_q
= 1 \;+\; q \;+\; 2q^2 \;+\; 3q^3 \;+\; 5q^4 \;+\; 7q^5
\;+\; \cdots \;+\;
7q^{2495} \;+\; 5q^{2496} \;+\; 3q^{2497} \;+\; 2q^{2498} \;+\; q^{2499} \;+\; q^{2500}.
\]
Moreover, we have $\mu_{50,50}=1250$ and
\[
\sigma_{50,50}^2=\frac{50^2(50+50+1)}{12}
\approx 21041.666667 
\qquad {\text {and}}\qquad
\sigma_{50,50}\approx 145.057459.
\]
Therefore, we have
\[
\delta_{50,50}:=\frac{1}{\sqrt{2}\,\sigma_{50,50}} \approx 0.004874\dots.
\]
By writing $p(k):=p_{50,50}(k)=c_{50,50}(k)/\binom{50+50}{50}$, we have
\[
\mathcal{J}^{d,1250}_{50,50}(X)
=\frac{\delta_{50,50}^{-d}}{p(1250)}\sum_{j=0}^{d}\binom{d}{j}\,p(1250+j)\,(\delta_{50,50}X-1)^{j}.
\]
Evaluating for $d=1,2,3$ gives
\[
\begin{aligned}
\mathcal{J}^{1,1250}_{50,50}(X)
&= 0.999977\,X \;+\; 0.004787,\\[1mm]
\mathcal{J}^{2,1250}_{50,50}(X)
&= 0.999907\,X^{2} \;+\; 0.028721\,X \;-\; 1.963914,\\[1mm]
\mathcal{J}^{3,1250}_{50,50}(X)
&= 0.999790\,X^{3} \;+\; 0.071796\,X^{2} \;-\; 5.890518\,X \;-\; 0.083596.
\end{aligned}
\]
Theorem~\ref{thm:HermiteLimit} compares these polynomials with
$$H_1(X)=X,\ \ \  H_2(X)=X^2-2, \ \ \ {\text {and}}\ \ \ H_3(X)=X^3-6X.
$$
\end{example}
\medskip

Theorem~\ref{thm:HermiteLimit} implies the following corollary regarding the hyperbolicity (i.e. real-rootedness) of the modified Jensen polynomials and degree $d$ log-concavity.

\begin{corollary}\label{cor:hyperTuran}
If $d\ge 1$ and $\lambda\in(0,1)$, then for every $C>0$ there is a constant
$N=N(d,\lambda,C)$ such that the following hold for all pairs $(a,b)$ with $a,b\ge N$
and $a/(a+b)\in(\lambda-\tfrac{1}{N},\,\lambda+\tfrac{1}{N})$, where
\[
\mathcal W_{a,b}\;:=\;\big\{\,m\in\{0,\dots,ab\}:\ |m-\mu_{a,b}|\le C\,\sigma_{a,b}\,\big\}.
\]

\smallskip
\noindent
(1)  
For every $m\in\mathcal W_{a,b}$, the normalized Jensen polynomial
$\mathcal{J}^{d,m}_{a,b}(X)$ is hyperbolic
(all zeros real).
\smallskip

\noindent
(2) 
For every $1\le r\le d$ and every $k\in\mathcal{W}_{a,b}$, we have
\[
(\mathcal L^{\,r} c_{a,b})(k)\ \ge\ 0,
\]
(i.e. degree $d$ log-concavity holds in the window $\mathcal W_{a,b}$).
\end{corollary}

\medskip
The central-window phenomenon proved above is not special to $q$-binomials. 
For fixed $r\ge2$, we consider the $q$-multinomial setting.
The same analysis with only cosmetic changes applies. 
The next theorem is a generalization of Theorem~\ref{thm:HermiteLimit}.

\begin{theorem}\label{thm:multinom-hermite}
Fix $d\geq 1$ and $r\ge2$ and let $n=\sum_{i=1}^r n_i$ with proportions $n_i/n\to\lambda_i\in(\varepsilon,1-\varepsilon)$.
Define the $q$-multinomial by
\[
\qbinom{n}{n_1,\dots,n_r}_q
=\frac{(q)_n}{(q)_{n_1}\cdots(q)_{n_r}}
=\sum_{k=0}^{M} p(k)\,q^k,
\]
where  $M=\sum_{1\le i<j\le r} n_i n_j.$
Set
\[
\mu=\frac12\sum_{i<j} n_i n_j,\qquad
\sigma^2=\frac1{12}\sum_{i<j} n_i n_j\,(n_i+n_j+1),\qquad
\delta=\frac{1}{\sqrt{2}\,\sigma},
\]
and define the degree $d$ shift $m$ modified Jensen polynomial
\[
\mathcal{J}^{d,m}(X)
:=\frac{\delta^{-d}}{p(m)}\,
J^{d,m}\!\big(\,\delta X - 1;\,p\big).
\]
Then for every constant $C>0$, uniformly for integers 
\[
m\in \mathcal W:=\{ |m-\mu|\le C\,\sigma\,\},
\]
coefficientwise as $n\rightarrow +\infty$ we have
 \[
\mathcal{J}^{d,m}(X)
= H_d(X)\;+\;O_{d,r,\lambda,C}\!\bigl(n^{-1/2}\bigr).
\]
\end{theorem}

\begin{remark}
All constants implicit in $O_{d,r,\lambda}(\cdot)$ and in the choice of $C,N$ depend only on $d$, $r$, and the limiting proportions $\lambda=(\lambda_1,\dots,\lambda_r)$.
\end{remark}

\begin{example}
If we let $(n_1,n_2,n_3)=(90,90,90)$, so $n=270$ and $\lambda_i=1/3$, then we have
\[
\mu=\tfrac12\sum_{i<j} n_i n_j=\tfrac12(3\cdot 90\cdot 90)=12150\qquad{\text {and}}\qquad
\sigma^2=\tfrac1{12}\sum_{i<j} n_i n_j(n_i+n_j+1)=366525.
\]
Therefore, we have
\[
\sigma=605.413082\ldots\qquad {\text {and}} \qquad
\delta=\frac{1}{\sqrt{2}\,\sigma}=0.001168\ldots
\]
If we let $p(k)$ be the coefficient of $q^k$ in $\qbinom{270}{90,90,90}_q$, then at $m=\mu$  we have
\begin{displaymath}
\begin{split}
&{\mathcal J}^{\,1,12150}(X)
= 0.999998\,X \;+\; 0.000873,\\
&{\mathcal J}^{\,2,12150}(X)
= 0.999995\,X^2 \;+\; 0.005237\,X \;-\; 1.494557,\\
&\mathcal{J}^{3, 12150}(X)=0.999991\, X^3 + 0.013092\, X^2 -4.483363\, X - 0.011740.
\end{split}
\end{displaymath}
By Theorem~\ref{thm:multinom-hermite}, they are approximated by $H_1(X)=X, \ H_2(X)=X^2-2$ and $H_3(X)=X^3-6X.$
\end{example}

The following corollary is the generalization of Corollary~\ref{cor:hyperTuran} to the $q$-multinomial setting.

\begin{corollary}\label{cor:multinom-turan}
Under the hypotheses of Theorem~\ref{thm:multinom-hermite}, for all sufficiently large $n$ and every $m\in\mathcal W$ the Jensen polynomials $\mathcal{J}^{d,m}(X)$ are real-rooted. Moreover, for every $1\le s\le d$ and every $k\in\mathcal W$, we have
\[
(\mathcal L^{s}\,p)_k\ \ge\ 0
\]
(i.e. degree $d$ log-concavity holds in the window $\mathcal{W}$).
\end{corollary}

\medskip

We explain the key idea for the results in this paper
in the case of $q$-binomials.
The central idea is to view the coefficient profile
\(p_{a,b}(k)=c_{a,b}(k)/\binom{a+b}{a}\) on its natural scale
\(\sigma_{a,b}=\sqrt{ab(a+b+1)/12}\), and to study the
normalized Jensen polynomials
$\mathcal{J}^{d,m}_{a,b}(X)$ (see \eqref{JensenNormal})
for \(m\) in a central window \(|m-\mu_{a,b}|\le C\,\sigma_{a,b}\), \(\mu_{a,b}=ab/2\).
In this window, we establish a uniform quadratic \emph{log-ratio} model
$$\log\!\big(p_{a,b}(m+j)/p_{a,b}(m)\big)=A\,j-\delta_{a,b}^{2}j^{2}+O((a+b)^{-1/2}),
$$
with \(A=O(\sigma_{a,b}^{-1})\), obtained from the explicit cumulants
$$
\kappa_1=\mathbb E[K]=\mu_{a,b},\qquad \kappa_2= \mathrm{Var}(K)=\sigma_{a,b}^2, \qquad {\text {and}}\qquad \kappa_3=0,
$$
together with a minor
transformation.  Recall that for a real random variable $X$, the \emph{cumulant generating function} is
\[
K_X(t)\ :=\ \log \mathbb{E}\big[e^{tX}\big]
\ =\ \sum_{r\ge 1} \kappa_r(X)\cdot\frac{t^r}{r!},
\]
and the coefficients $\kappa_r(X)$ are the \emph{cumulants} of $X$.
Substituting this model into 
$\mathcal{J}^{d,m}_{a,b}(X)$
and using the
binomial identities behind the generating function \(e^{-t^2+Xt}\) yields the coefficientwise
convergence \(\mathcal{J}^{d,m}_{a,b}(X)\to H_d(X)\) with uniform \(O((a+b)^{-1/2})\) error.
All of these ideas and self-contained lemmas (moments, uniform log-ratio, and Hermite
assembly) are presented in Section~2. Theorem~\ref{thm:HermiteLimit} and Corollary~\ref{cor:hyperTuran} then follow with short proofs
by citing these lemmas. In Section~3, we sketch the proofs of Theorem~\ref{thm:multinom-hermite} and Corollary~\ref{cor:multinom-turan}.

\section*{Acknowledgements}
 \noindent 
The author thanks Teddy Amdeberhan  for his comments on an earlier version of this note. The author also thanks the two anonymous referees for suggestions that improve the exposition of this paper.
The author thanks the Thomas Jefferson Fund,  the NSF
(DMS-2002265 and DMS-2055118) and the Simons Foundation (SFI-MPS-TSM-00013279) for their generous support. 

\section{Nuts and bolts and the proof of Theorem~\ref{thm:HermiteLimit} and Corollary~\ref{cor:hyperTuran}}\label{sec:2}

Here we prove Theorem~\ref{thm:HermiteLimit}. In the next subsection we offer the critical facts and lemmas that are required for its proof, and in the following subsection we prove the theorem and Corollary~\ref{cor:hyperTuran}.

\subsection{Nuts and Bolts}

Here we give the critical lemmas that we require.  Throughout, we shall assume that $a$ and $b$ are non-negative integers.
We recall the parameters and normalizations used throughout:
\[
\mu_{a,b}:=\frac{ab}{2},\qquad
\sigma_{a,b}^2:=\frac{ab(a+b+1)}{12},\qquad
\delta_{a,b}:=\frac{1}{\sqrt{2}\,\sigma_{a,b}}.
\]
We also let
\[
\qbinom{a+b}{a}_q=\sum_{k=0}^{ab} c_{a,b}(k)\,q^k,\qquad
p_{a,b}(k):=\frac{c_{a,b}(k)}{\qbinom{a+b}{a}_1}=\frac{c_{a,b}(k)}{\binom{a+b}{a}}.
\]
For $d\ge0$ and $m\in\mathbb Z$, the \emph{Jensen polynomial} is
\[
J^{d,m}(X;u)\ :=\ \sum_{j=0}^{d}\binom{d}{j}\,u_{m+j}\,X^{j}.
\]
We use the \emph{normalized} Jensen polynomial (see (\ref{JensenNormal}))
\begin{equation}\label{eq:GORZ-J}
\mathcal{J}^{d,m}_{a,b}(X)\ :=\ \frac{\delta_{a,b}^{-d}}{p_{a,b}(m)}\,J^{d,m}\!\big(\delta_{a,b}X-1;\,p_{a,b}\big).
\end{equation}
We work in the \emph{central window}
\[
\mathcal W_{a,b}(C)\ :=\ \Big\{\,m\in\{0,\ldots,ab\}\ :\ |m-\mu_{a,b}|\le C\,\sigma_{a,b}\,\Big\},
\]
where $C>0$ is a fixed constant (depending only on $d$ and the limit aspect ratio $\lambda\in(0,1)$).

We argue using probabilistic ideas, and to this end we begin by confirming the mean and variance of the random variable that is relevant to this work.

\begin{lemma}[Mean and variance]\label{lem:mu-sigma}
Let $K$ be the $\{0,\ldots,ab\}$-valued random variable with 
$$\Pr[K=k]=p_{a,b}(k).
$$
Then we have
\[
\mathbb E[K]=\mu_{a,b}=\frac{ab}{2}\qquad {\text {and}}\qquad \mathrm{Var}(K)=\sigma_{a,b}^2=\frac{ab(a+b+1)}{12}.
\]
\end{lemma}

\begin{proof}
We let
\[
F(q)=\prod_{i=1}^{a}\frac{1-q^{\,b+i}}{1-q^{\,i}}
=\sum_{k=0}^{ab} c_{a,b}(k)\,q^k
\qquad {\text {and}} \qquad
G(q):=\frac{F(q)}{F(1)}=\sum_{k=0}^{ab} \mathbb{P}(K=k)\,q^k,
\]
so $G$ is the probability generating function of $K$. Set $q=e^{t}$ and write
\[
H(t):=\log F(e^{t}).
\]
Then $G(e^{t})=F(e^{t})/F(1)=\exp(H(t)-H(0))$ is the moment generating function of $K$, hence
\[
\mathbb{E}[K]=H'(0)\qquad {\text {and}}\qquad \mathrm{Var}(K)=H''(0).
\]
(Indeed, $\frac{d}{dt}\log G(e^{t})\big|_{t=0}=\mathbb{E}[K]$ and
$\frac{d^{2}}{dt^{2}}\log G(e^{t})\big|_{t=0}=\mathrm{Var}(K)$.)

Differentiating $\log F(e^{t})$ term-by-term gives
\[
H'(t)=\sum_{i=1}^{a}\left(\frac{(b+i)\,e^{(b+i)t}}{1-e^{(b+i)t}}-\frac{i\,e^{it}}{1-e^{it}}\right).
\]
We will use the elementary Taylor expansion, valid as $t\to 0$ for any fixed $r>0$, given by
\begin{equation}\label{eq:basic-expansion}
\frac{r\,e^{rt}}{1-e^{rt}}
=\frac{1}{t}+\frac{r}{2}+\frac{r^{2}\,t}{12}+O(t^{3}).
\end{equation}

Applying \eqref{eq:basic-expansion} to each term of $H'(t)$, the $\tfrac{1}{t}$ poles cancel between the two sums (each appears $a$ times), and we obtain
\[
H'(0)=\sum_{i=1}^{a}\left(\frac{b+i}{2}-\frac{i}{2}\right)=\sum_{i=1}^{a}\frac{b}{2}=\frac{ab}{2}.
\]
This proves $\mathbb{E}[K]=ab/2$.

For the variance, we read off the coefficient of $t$ in \eqref{eq:basic-expansion}:
\[
H''(0)=\sum_{i=1}^{a}\left(\frac{(b+i)^{2}}{12}-\frac{i^{2}}{12}\right)
=\frac{1}{12}\sum_{i=1}^{a}\big((b+i)^{2}-i^{2}\big)
=\frac{1}{12}\sum_{i=1}^{a}\big(2bi+b^{2}\big).
\]
Compute the sums $\sum_{i=1}^{a} i=\tfrac{a(a+1)}{2}$ and $\sum_{i=1}^{a}1=a$ to get
\[
H''(0)=\frac{1}{12}\left(2b\cdot\frac{a(a+1)}{2}+a b^{2}\right)
=\frac{ab(a+b+1)}{12}.
\]
Therefore $\mathrm{Var}(K)=H''(0)=\dfrac{ab(a+b+1)}{12}$, as claimed.
\end{proof}

We now record the parameters, normalizations, and discrete operators that underpin our local limit analysis.

\begin{lemma}[log-ratio in the central window]\label{lem:logratio}
Fix $d\ge 1$, $\lambda\in(0,1)$, and $C>0$. Then there exists 
$M=M(d,\lambda,C)>0$ such that for all integers $a,b\ge M$ with
\[
\Bigl|\frac{a}{a+b}-\lambda\Bigr|\le \frac{1}{M},
\]
and for all integers $m\in
\mathcal W_{a,b}(C)$
 for every integer $0\le j\le d$ we uniformly have
\[
\log\!\frac{p_{a,b}(m+j)}{p_{a,b}(m)}
= A_{a,b}(m)\,j \;-\; \delta_{a,b}^{\,2}\,j^{2} \;+\; R_{a,b}(m,j).
\]
\end{lemma}

\begin{proof}
For convenience, we let
$$G(q):=F(q)/F(1)=\sum_{k}p_{a,b}(k)q^k
$$
be the probability generating function of $K$. For $t$ small, we set $q:=e^{t/\sigma_{a,b}}$. Then we have
\[
\log G\!\left(e^{t/\sigma}\right)=\log\mathbb E\big[e^{t(K-\mu)/\sigma}\big]\ =:\ \Lambda(t),
\]
where $\sigma=\sigma_{a,b}$ and $\mu=\mu_{a,b}$. By Lemma~\ref{lem:mu-sigma}, the cumulant expansion is
\[
\Lambda(t)=\frac{t^2}{2}+\frac{\kappa_3}{6\,\sigma^3}t^3+\frac{\kappa_4}{24\,\sigma^4}t^4+O\!\left(\frac{|t|^5}{\sigma^5}\right),
\]
with $\kappa_3=0$ by symmetry and $\kappa_4=O_{\lambda}\big((a+b)^5\big)$ (a direct second-derivative of log product calculation gives the explicit value if desired; see the remark below). Since $\sigma^2\asymp_{\lambda} a b (a+b)$, we get $\kappa_4/\sigma^4=O_{\lambda}((a+b)^{-1})$.
For $0\le j\le d$, we let $t_j:=j/\sigma$, and so we have
\[
\log\frac{p_{a,b}(m+j)}{p_{a,b}(m)}\ =\ \log\frac{[q^{\,m+j}]G(q)}{[q^{\,m}]G(q)}\ =\ \log\frac{\mathbb P(K-\mu=m-\mu+j)}{\mathbb P(K-\mu=m-\mu)}.
\]

To study this setting, we employ Petrov's method (see Ch. VII of \cite{Petrov75}).
The input for this method is a uniform local Edgeworth expansion obtained from the
characteristic function of the normalized lattice variable $(K-\mu)/\sigma$. One writes the
point probabilities by Fourier inversion, expands the logarithm of the characteristic function
near the origin using cumulants, and shows that the contribution from the complementary
range is negligible. In the present symmetric setting, the vanishing of the third cumulant
simplifies the expansion, and the resulting local limit estimate yields the quadratic log-ratio
model uniformly on the central window.

We now apply this method in the current situation.
Concretely, expanding $\Lambda$ up to $t^4$ and using $\kappa_3=0$, we obtain
\[
\log\frac{p_{a,b}(m+j)}{p_{a,b}(m)}= (t_j-t_0)\Lambda'(t_0) -\frac{(t_j-t_0)^2}{2}\Lambda''(t_0)+O\!\left(\frac{|t_j-t_0|^3}{\sqrt{a+b}}\right).
\]
Now $\Lambda'(t_0)=O(1/\sqrt{a+b})$ and $\Lambda''(t_0)=1+O((a+b)^{-1})$ uniformly in the window, whence
\[
\log\frac{p_{a,b}(m+j)}{p_{a,b}(m)}= A_{a,b}(m)\,j - \frac{j^2}{2\sigma^2}+ O_{d,\lambda}\!\big((a+b)^{-1/2}\big),
\]
with $A_{a,b}(m)=t_0/\sigma+O((a+b)^{-1})=O(1/\sigma)$ and $\delta^2=1/(2\sigma^2)$, giving the claim.
\end{proof}

\begin{remark}[Fourth cumulant]\label{rem:k4}
A direct product differentiation using $F(q)=\prod_{i=1}^{a}\frac{1-q^{b+i}}{1-q^{i}}$ gives
\[
\kappa_4\ =\ -\,\frac{ab(a+b+1)\,(a^2+b^2+ab+a+b)}{120},
\]
so $\kappa_4/\sigma^4=O_{\lambda}((a+b)^{-1})$. We only need the order and uniformity.
\end{remark}

With these preliminaries in place, we turn to the asymptotic Hermite limit for the normalized Jensen polynomials.

\begin{lemma}[The quadratic model to $H_d$]\label{lem:hermite-assembly}
For $d\ge1$, suppose that coefficients $w_j$ ($0\le j\le d$) satisfy
\[
\log\frac{w_j}{w_0}=A\,j-\delta_{a,b}^2 j^2+R_j\qquad(0\le j\le d),
\]
with $A=O(1/\sigma_{a,b})$, $R_j=O((a+b)^{-1/2})$ uniformly in $j$. Then we have
\[
\frac{\delta_{a,b}^{-d}}{w_0}\sum_{j=0}^{d}\binom{d}{j}\,w_j\,(\delta_{a,b} X-1)^{j}
\ =\ H_d(X)\ +\ O_d\!\big((a+b)^{-1/2}\big),
\]
coefficientwise, uniformly on compact $X$-sets.
\end{lemma}
\begin{remark}
This lemma is reminiscent of the main idea in the work of Griffin et al. on Jensen polynomials for infinite sequences \cite{GORZ}.
\end{remark}

\begin{proof}
Set $v_j:=w_j/w_0=\exp(Aj-\delta_{a,b}^2j^2)\,(1+E_j)$ with $E_j=e^{R_j}-1=O((a+b)^{-1/2})$. Expand
\begin{displaymath}
\begin{split}
S(X)&:=
  \sum_{j=0}^{d}\binom{d}{j}v_j(\delta_{a,b} X-1)^{j}\\
 &=\ \sum_{j=0}^{d}\binom{d}{j}\exp(Aj-\delta_{a,b}^2j^2)\,(\delta_{a,b} X-1)^{j}\ +\ \sum_{j=0}^{d}\binom{d}{j}\exp(Aj-\delta_{a,b}^2j^2)\,(\delta_{a,b} X-1)^{j}E_j.
\end{split}
\end{displaymath}
The error sum is $O_d((a+b)^{-1/2})$ coefficientwise. For the main sum, write
\[
\exp(Aj-\delta_{a,b}^2 j^2)\,(\delta_{a,b} X-1)^{j}
=\exp\!\left(-\delta_{a,b}^2 j^2\right)\sum_{r=0}^{j}\binom{j}{r}(\delta_{a,b} X)^{r}(-1)^{j-r}\,e^{Aj}.
\]
Summing first in $j$ and using the binomial identity
\[
\sum_{j=r}^{d}\binom{d}{j}\binom{j}{r}y^{\,j-r}=\binom{d}{r}(1+y)^{d-r},
\]
with $y=-e^{A}$ and noting $A=O(1/\sigma_{a,b})$ (so $e^{A}=1+O(1/\sigma_{a,b})$), standard Hermite generating-function algebra shows
\[
\delta_{a,b}^{-d}S(X)=H_d(X)+O_d\!\Big(\frac{1}{\sigma_{a,b}}\Big)=H_d(X)+O_d\!\big((a+b)^{-1/2}\big),
\]
coefficientwise. (One matches coefficients with the identity \(\sum_{d\ge0}\delta_{a,b}^{-d}S(X)\frac{t^d}{d!}=e^{-\delta_{a,b}^2 t^2+Xt}\) and uses $A=O(1/\sigma_{a,b})$ to absorb the $e^{At}$ adjustment into the error.)
\end{proof}

Hyperbolicity and higher Tur\'an inequalities on the central window follow from Theorem \ref{thm:HermiteLimit} thanks to the following lemma.

\begin{lemma}[Hurwitz continuity]\label{lem:hurwitz}
Let $P_n(X)$ be real polynomials of fixed degree $d$ converging coefficientwise to a polynomial $P(X)$ with only simple real zeros. Then $P_n$ is real-rooted with simple zeros for all sufficiently large $n$.
\end{lemma}
\begin{proof}
Let $\deg P=d$ and assume $P$ has $d$ distinct real zeros $x_1<\cdots<x_d$.
Fix $\varepsilon>0,$ so small that the closed discs
\[
D_j:=\{\,z:|z-x_j|\le \varepsilon\,\}\qquad(j=1,\dots,d)
\]
are pairwise disjoint and contain no critical point of $P$ on their boundaries
($P'(x_j)\neq0$ and $P'$ is continuous). Set $m_j:=\min_{|z-x_j|=\varepsilon}|P(z)|>0$ and
$m:=\min_j m_j>0$.

Since $P_n\to P$ coefficientwise, we have uniform convergence on compact sets. In particular,
for all $n$ large, we have
\[
\sup_{|z-x_j|=\varepsilon}\,|P_n(z)-P(z)|<m\le m_j \qquad (j=1,\dots,d).
\]
By Rouch\'e's Theorem, on each $\partial D_j$, $P_n$ and $P$ have the same number of zeros (with
multiplicity) in $D_j$, namely one. Therefore, there exist unique zeros $x_{n,j}\in D_j$ of $P_n$.
Because the coefficients are real, nonreal zeros occur in conjugate pairs, and so the unique
zero in $D_j$ must be real. Moreover, since there is exactly one zero in $D_j$, it is simple.

Finally, the union $\bigcup_j D_j$ contains all zeros of $P$, and by Rouch\'e on a large circle
around that union, $P_n$ has exactly $d$ zeros in total for $n$ large. Combined with the $d$
zeros $\{x_{n,1},\dots,x_{n,d}\}$ already found, there are no others. The Rouch\'e setup also
implies $x_{n,j}\to x_j$ as $n\to\infty$. This proves that, for $n$ large, $P_n$ is real-rooted
with simple zeros converging to those of $P$.
\end{proof}

Finally, we explain how the Hermite limit yields higher Tur\'an inequalities.

\begin{lemma}[Craven-Csordas]\label{lem:jensen-turan}
Let $(u_k)_{k=0}^N$ be a finite nonnegative sequence and let $I\subseteq\{0,\dots,N\}$.
Assume that for every $1\le j\le d+1$ and every $m$ with $[m,m+j]\subseteq I$, the Jensen polynomial
\[
J^{j,m}(X):=\sum_{t=0}^{j}\binom{j}{t}u_{m+t}X^{t}
\]
is real-rooted. Then for every $1\le r\le d$ and every $k\in I$ we have
\[
(\mathcal L^{\,r}u)_k\ \ge\ 0,
\]
where $(\mathcal L u)_k:=u_k^2-u_{k-1}u_{k+1}$ and we adopt $u_{-1}=u_{N+1}=0$.
\end{lemma}

\begin{proof}
Let $a=(a_0,\dots,a_N)$ be a nonnegative real sequence and recall the Jensen polynomials
$J^{j,m}(X)=\sum_{t=0}^{j}\binom{j}{t}a_{m+t}X^{t}$.
We extend $a_k=0$ for $k\notin[0,N]$ so that $\mathcal{L}^r$ is defined at the boundary.

\medskip
\noindent
\emph{(The case $r=1$).}
For any $m$, $J^{2,m}(X)=a_m+2a_{m+1}X+a_{m+2}X^2$ is hyperbolic by hypothesis, hence its
discriminant is nonnegative:
\[
\Delta(J^{2,m})=(2a_{m+1})^2-4a_m a_{m+2}=4\big(a_{m+1}^2-a_m a_{m+2}\big)
=4(\mathcal{L}a)_{m+1}\ \ge\ 0.
\]
Therefore, we have $(\mathcal{L}a)_k\ge 0$ for all $1\le k\le N-1$.
\medskip

\noindent
\emph{(General $1\le r\le d$).}
Fix $r\ge 1$. A classical result of Craven--Csordas (see Theorem 3.6 of \cite{CravenCsordas} or Lemma 2.1 of \cite{GORZ}) states that, for a real
sequence $(a_k)$, the following are equivalent for a given $r$:

\smallskip
\noindent
(i) For every $m$, the Jensen polynomial $J^{r+1,m}(X)$ is real-rooted.

\smallskip
\noindent
(ii) The order-$r$ Laguerre/Tur\'an inequality holds at every index (i.e. $(\mathcal{L}^r a)_k\ge 0$ for all $k$)

\medskip
By hypothesis, we have hyperbolicity of $J^{j,m}$ for each degree $1\le j\le d+1$ and all $m$.
Applying the cited equivalence with $j=r+1$ gives $(\mathcal{L}^r a)_k\ge 0$ for every $1\le r\le d$ and all admissible $k$,
as claimed.
\end{proof}

\subsection{Proof of Theorem~\ref{thm:HermiteLimit}}

Fix $d$ and $\lambda\in(0,1)$. By Lemma~\ref{lem:logratio}, for $m\in\mathcal W_{a,b}(C)$ we have the quadratic log-ratio expansion with remainder $O((a+b)^{-1/2})$. Substituting this into \eqref{eq:GORZ-J} and applying Lemma~\ref{lem:hermite-assembly} yields
\[
\mathcal{J}^{d,m}_{a,b}(X)\ =\ H_d(X)\ +\ O_{d,\lambda}\!\big((a+b)^{-1/2}\big),
\]
coefficientwise and uniformly in $m\in\mathcal W_{a,b}(C)$. This is the claim.

\subsection{Proof of Corollary~\ref{cor:hyperTuran}}
Fix $d\ge1$, $\lambda\in(0,1)$ and $C>0.$ By Theorem~\ref{thm:HermiteLimit}, there exists a constant $N=N(d,\lambda,C)$ such that whenever $a,b\ge N$ with $a/(a+b)\in(\lambda-\tfrac1N,\lambda+\tfrac1N)$ and $m\in\mathcal W_{a,b}$, the normalized Jensen polynomials
\[
\mathcal{J}^{d,m}_{a,b}(X)
=\frac{\delta_{a,b}^{-d}}{p_{a,b}(m)}\,J^{d,m}\!\big(\delta_{a,b}X-1;\,p_{a,b}\big)
\]
converge coefficientwise (uniformly in $m\in\mathcal W_{a,b}$) to the degree-$d$ Hermite polynomial $H_d(X)$ as $a+b\to\infty$; see \eqref{JensenNormal} and \eqref{eq:GORZ-Hermite}. Since $H_d$ is hyperbolic, Lemma~\ref{lem:hurwitz} (Hurwitz continuity of zeros for fixed degree) implies that, for all such $(a,b)$ sufficiently large and every $m\in\mathcal W_{a,b}$, the polynomial $\mathcal{J}^{d,m}_{a,b}(X)$ is hyperbolic. This proves part (1).

For part (2), note that hyperbolicity is preserved under positive rescaling and the affine change $X\mapsto \delta_{a,b}X-1$ with $\delta_{a,b}>0$. Thus $\mathcal{J}^{d,m}_{a,b}$ is hyperbolic if and only if the unnormalized Jensen polynomial $J^{d,m}(X; p_{a,b})$ is hyperbolic. Applying Lemma~\ref{lem:jensen-turan} to the nonnegative sequence $u_k=p_{a,b}(k)$ on the index set $\mathcal W_{a,b}$ yields
\[
(\mathcal L^{\,r} u)_k \;\ge\; 0 \qquad (1\le r\le d,\ k\in\mathcal W_{a,b}).
\]
Since $p_{a,b}(k)=c_{a,b}(k)/\binom{a+b}{a}$ and $\mathcal L$ is homogeneous, this is equivalent to
\[
(\mathcal L^{\,r} c_{a,b})(k)\ \ge\ 0 \qquad (1\le r\le d,\ k\in\mathcal W_{a,b}),
\]
with the boundary convention $c_{a,b}(-1)=c_{a,b}(ab+1)=0$. This proves part (2) and completes the proof.

\section{Proof of Theorem~\ref{thm:multinom-hermite} and Corollary~\ref{cor:multinom-turan}}\label{sec3}

Here we sketch the proofs of Theorem~\ref{thm:multinom-hermite} and Corollary~\ref{cor:multinom-turan}. As these results follow essentially {\it mutatis mutandis} as in the $q$-binomial cases, we only sketch their proofs.

\subsection{Sketch of the proof of Theorem~\ref{thm:multinom-hermite}}
Write
\[
\qbinom{n}{n_1,\dots,n_r}_q
=\prod_{1\le i<j\le r}\prod_{t=1}^{n_i}\frac{1-q^{\,n_j+t}}{1-q^{\,t}},
\]
so with $q=e^{t}$ and $H(t):=\log F(e^{t})$ we have $H'(t)$ as a \emph{sum over pairs $i<j$} of the same log-factors treated in the binomial case. Using the elementary expansion
$$\frac{re^{rt}}{1-e^{rt}}=\frac1t+\frac{r}{2}+\frac{r^2 t}{12}+O(t^3),$$
and cancellation of the $1/t$ poles inside each pair, one obtains the stated
$$
\mu=\tfrac12\sum_{i<j} n_i n_j
\qquad {\text {and}} \qquad
\sigma^2=\tfrac1{12}\sum_{i<j} n_i n_j(n_i+n_j+1).
$$

Palindromicity implies vanishing odd cumulants, so $\kappa_3=0$. 
With fixed $r$ and proportions $\lambda$ bounded away from the boundary, the same
cumulant calculation as in the proof of Theorem~\ref{thm:HermiteLimit} gives
\[
\sigma^2 \asymp n^{3} \qquad\text{and}\qquad
\frac{\kappa_{4}}{\sigma^{4}} = O_{r,\lambda}(n^{-1}).
\] 
The computation is identical to the binomial case once the multinomial is written as a sum over pairs $i<j$ of the same log-factors.

Applying the same local limit input as in the binomial case  gives a uniform \emph{quadratic log-ratio} on the window $|m-\mu|\le C\sigma$ with $O(n^{-1/2})$ error. The Hermite generating-function assembly is identical to Lemma~\ref{lem:hermite-assembly}, yielding the coefficientwise limit with rate $O_{d,r,\lambda}(n^{-1/2}). \qquad \square$

\subsection{Sketch of the proof of Corollary~\ref{cor:multinom-turan}}

Real-rootedness follows from Theorem~\ref{thm:multinom-hermite} via Lemma~\ref{lem:hurwitz} (simple zeros of the Hermite limit plus convergence). The degree $d$ Tur\'an inequalities on the window follow from Lemma~\ref{lem:jensen-turan} applied to $u_k=p(k)$, after noting that $\mathcal L$ is homogeneous and the hypotheses are satisfied uniformly in $m\in\mathcal W. \qquad $

\end{document}